\documentclass[11pt]{amsart}
\usepackage{amssymb,amsthm,amsmath,epsfig,latexsym}
\usepackage{calc,times,verbatim}
\usepackage{geometry} 
\begin{document}

\newcommand{\mmbox}[1]{\mbox{${#1}$}}
\newcommand{\proj}[1]{\mmbox{{\mathbb P}^{#1}}}
\newcommand{\Cr}{C^r(\Delta)}
\newcommand{\CR}{C^r(\hat\Delta)}
\newcommand{\affine}[1]{\mmbox{{\mathbb A}^{#1}}}
\newcommand{\Ann}[1]{\mmbox{{\rm Ann}({#1})}}
\newcommand{\caps}[3]{\mmbox{{#1}_{#2} \cap \ldots \cap {#1}_{#3}}}
\newcommand{\Proj}{{\mathbb P}}
\newcommand{\N}{{\mathbb N}}
\newcommand{\Z}{{\mathbb Z}}
\newcommand{\R}{{\mathbb R}}
\newcommand{\A}{{\mathcal{A}}}
\newcommand{\Tor}{\mathop{\rm Tor}\nolimits}
\newcommand{\Ext}{\mathop{\rm Ext}\nolimits}
\newcommand{\Hom}{\mathop{\rm Hom}\nolimits}
\newcommand{\im}{\mathop{\rm Im}\nolimits}
\newcommand{\rank}{\mathop{\rm rank}\nolimits}
\newcommand{\supp}{\mathop{\rm supp}\nolimits}
\newcommand{\arrow}[1]{\stackrel{#1}{\longrightarrow}}
\newcommand{\CB}{Cayley-Bacharach}
\newcommand{\coker}{\mathop{\rm coker}\nolimits}
\sloppy
\theoremstyle{plain}

\newtheorem{defn0}{Definition}[section]
\newtheorem{prop0}[defn0]{Proposition}
\newtheorem{quest0}[defn0]{Question}
\newtheorem{thm0}[defn0]{Theorem}
\newtheorem{lem0}[defn0]{Lemma}
\newtheorem{corollary0}[defn0]{Corollary}
\newtheorem{example0}[defn0]{Example}
\newtheorem{remark0}[defn0]{Remark}

\newenvironment{defn}{\begin{defn0}}{\end{defn0}}
\newenvironment{prop}{\begin{prop0}}{\end{prop0}}
\newenvironment{quest}{\begin{quest0}}{\end{quest0}}
\newenvironment{thm}{\begin{thm0}}{\end{thm0}}
\newenvironment{lem}{\begin{lem0}}{\end{lem0}}
\newenvironment{cor}{\begin{corollary0}}{\end{corollary0}}
\newenvironment{exm}{\begin{example0}\rm}{\end{example0}}
\newenvironment{rem}{\begin{remark0}\rm}{\end{remark0}}

\newcommand{\defref}[1]{Definition~\ref{#1}}
\newcommand{\propref}[1]{Proposition~\ref{#1}}
\newcommand{\thmref}[1]{Theorem~\ref{#1}}
\newcommand{\lemref}[1]{Lemma~\ref{#1}}
\newcommand{\corref}[1]{Corollary~\ref{#1}}
\newcommand{\exref}[1]{Example~\ref{#1}}
\newcommand{\secref}[1]{Section~\ref{#1}}
\newcommand{\remref}[1]{Remark~\ref{#1}}
\newcommand{\questref}[1]{Question~\ref{#1}}

\newcommand{\std}{Gr\"{o}bner}
\newcommand{\jq}{J_{Q}}


\title{On ideals generated by fold products of linear forms}
\author{\c{S}tefan O. Toh\v{a}neanu}

 \keywords{linear free resolution, ideals generated by products of linear forms, hyperplane arrangements, Orlik-Terao algebra, special fiber, symmetric ideal. \\ \indent Department of Mathematics, University of Idaho, Moscow, Idaho 83844-1103, USA, Email: tohaneanu@uidaho.edu, Phone: 208-885-6234, Fax: 208-885-5843.}

\begin{abstract}
Let $\mathbb K$ be a field of characteristic 0. Given $n$ linear forms in $R=\mathbb K[x_1,\ldots,x_k]$, with no two proportional, in one of our main results we show that the ideal $I\subset R$ generated by all $(n-2)$-fold products of these linear forms has linear graded free resolution. This result helps determining a complete set of generators of the symmetric ideal of $I$. Via Sylvester forms we can analyze from a different perspective the generators of the presentation ideal of the Orlik-Terao algebra of the second order; this is the algebra generated by the reciprocals of the products of any two (distinct) of the linear forms considered. We also show that when $k=2$, and when the collection of $n$ linear forms may contain proportional linear forms, for any $1\leq a\leq n$, the ideal generated by $a$-fold products of these linear forms has linear graded free resolution.
\end{abstract}
\maketitle

\section{Introduction}

Let $R:=\mathbb K[x_1,\ldots,x_k]$ be the ring of (homogeneous) polynomials with coefficients in $\mathbb K$, a field of characteristic 0, with the natural grading. Denote ${\frak m}:=\langle x_1,\ldots,x_k\rangle$ to be the irrelevant maximal ideal of $R$. Let $\ell_1,\ldots,\ell_n$ be linear forms in $R$, some possibly proportional, and denote this collection by $\Sigma=(\ell_1,\ldots,\ell_n)\subset R$; we need a notation to take into account the fact that some of these linear forms are proportional. For $\ell\in \Sigma$, by $\Sigma\setminus\{\ell\}$ we will understand the collection of linear forms of $\Sigma$ from which $\ell$ has been removed. Also, we denote $|\Sigma|=n$.

Let $1\leq a\leq n$ be an integer and define {\em the ideal generated by $a$-fold products of $\Sigma$} to be the ideal of $R$ $$I_a(\Sigma):=\langle \{\ell_{i_1}\cdots\ell_{i_a}|1\leq i_1<\cdots<i_a\leq n\}\rangle.$$ We also make the convention $I_0(\Sigma):=R$, and $I_b(\Sigma)=0$, for all $b>n$. Also, if $\Sigma=\emptyset$, $I_a(\Sigma)=0$, for any $a\geq 1$.

A homogeneous ideal $I\subset R$ generated in degree $d$ it is said to have {\em linear graded free resolution}, if one has the graded free resolution $$0\rightarrow R^{n_{b+1}}(-(d+b))\rightarrow\cdots\rightarrow R^{n_2}(-(d+1))\rightarrow R^{n_1}(-d)\rightarrow R \rightarrow R/I\rightarrow 0,$$ for some positive integer $b$. By convention, the zero ideal has linear graded free resolution.

Though only recently has been written down (see \cite[Conjecture 1]{AnGaTo}), it has been five years at least since it has been conjectured that for any $\Sigma\subset R$, and any $1\leq a\leq|\Sigma|$, the ideals $I_a(\Sigma)$ have linear graded free resolution.

Without any loss of generality, we can assume that $\langle \Sigma\rangle={\frak m}$; otherwise, after a change of variables we can assume that $\Sigma\subset\mathbb K[x_1,\ldots,x_s], s<k,$ with $\langle\Sigma\rangle=\langle x_1,\ldots,x_s\rangle$. Suppose $|\Sigma|=n$. Here are some instances when it is known that this conjecture is true:
\begin{itemize}
  \item[(1)] Let $\mathcal C_{\Sigma}$ be the linear code with generating matrix having columns dual to the linear forms of $\Sigma$ (in no particular order). This will be a linear code of length $n$, and dimension $k$. Suppose its minimum (Hamming) distance is $d$. Then, by \cite[Theorem 3.1]{To}, for any $1\leq a\leq d$, we have $I_a(\Sigma)={\frak m}^a$, and in these ideals have linear graded free resolution (see for example \cite[Corollary 1.5]{EiGo}).
  \item[(2)] With the same point of view from (1), for some $\Sigma$ with certain properties, in \cite[Theorem 3.1]{AnTo}, it is shown that $I_{d+1}(\Sigma)$ has linear graded free resolution (of course, after using \cite{EiGo}).
  \item[(3)] If any $k$ of the linear forms in $\Sigma$ are linearly independent (i.e., the linear code $\mathcal C_{\Sigma}$ is Maximum Distance Separable code), then $I_a(\Sigma)$ has linear graded free resolution, for any $1\leq a\leq n$. To see this, apply the proof of \cite[Theorem 2.5]{GaTo} and \cite[Proposition 3.5]{AnGaTo}.
  \item[(4)] More generally than part (3), whenever $R/I_a(\Sigma)$ is Cohen-Macaulay, then $I_a(\Sigma)$ has a linear graded free resolution. This can be seen from the discussions at the end of the proof of \cite[Proposition 2.1]{To2}, and immediately after it; the point there is that the Eagon-Northcott complex becomes the desired linear graded free resolution of $R/I_a(\Sigma)$ (see \cite[Theorem A2.60]{Ei}).
  \item[(5)] If no two of the linear forms of $\Sigma$ are proportional, then from part (4), $I_{n-1}(\Sigma)$, so $a=n-1$, has linear graded free resolution.
  \item[(6)] If $\langle \Sigma\rangle=\langle \ell\rangle$, for some $\ell\in R_1$, then for any $1\leq a\leq n$, we have $I_a(\Sigma)=\langle\ell^a\rangle$, which, as any principal ideal, has linear graded free resolution.
  \item[(7)] $I_n(\Sigma)$, so $a=n$, is the principal ideal generated by $\displaystyle \prod_{\ell\in\Sigma}\ell$, hence it has linear graded free resolution.
\end{itemize}

In this article we add to the list above three more nontrivial cases when the conjecture is true:

\begin{itemize}
    \item For any $k\geq 1$ and for any $\Sigma\subset R=\mathbb K[x_1,\ldots,x_k]$, a collection of $n$ linear forms with no two proportional, $I_{n-2}(\Sigma)$, so $a=n-2$, has linear graded free resolution. (see Theorem \ref{theorem1})
    \item If $k=2$, then for any $\Sigma\subset\mathbb K[x_1,x_2]_1$, and for any $1\leq a\leq |\Sigma|$, $I_a(\Sigma)$ has linear graded free resolution. (see Theorem \ref{theorem0})
    \item Generalizing part (5) above, for any $\Sigma\subset R=\mathbb K[x_1,\ldots,x_k]$, $I_{n-1}(\Sigma)$ has linear graded free resolution. (see Section \ref{a=n-1})
\end{itemize}

As applications to the main result Theorem \ref{theorem1}, we find a criterion when $R/I_{n-2}(\Sigma)$ is Cohen-Macaulay (Corollary \ref{CM}), and when $k=3$, we determine a primary decomposition of $I_{n-2}(\Sigma)$ (Proposition \ref{primaryDecomp}).

\medskip

Let $\Sigma=(\ell_1,\ldots,\ell_n)\subset R=\mathbb K[x_1,\ldots,x_k]$ be a collection of linear forms such that $\gcd(\ell_i,\ell_j)=1$, if $i\neq j$. Let $\A$ be the central hyperplane arrangement defined by $\ell_1,\ldots,\ell_n$, i.e., $\A=\{V(\ell_1),\ldots,V(\ell_n)\}\subset\mathbb P^{k-1}$. In such instance, instead of writing $\Sigma=(\ell_1,\ldots,\ell_n)\subset R$, we will write $\A=\{\ell_1,\ldots,\ell_n\}\subset R$. The {\em rank} of $\A$ is ${\rm rank}(\A)={\rm ht}(\langle\ell_1,\ldots,\ell_n\rangle)$; if ${\rm rank}(\A)=k$, then $\A$ is called essential.

In recent years, especially after the work of \cite{Te} and \cite{PrSp}, there has been a lot of focus on the algebra $OT(\A):=\displaystyle\mathbb K\left[\frac{1}{\ell_1},\ldots,\frac{1}{\ell_n}\right]$, called {\em the Orlik-Terao algebra} (after the names of the mathematicians who first introduced it in \cite{OrTe}), or {\em the algebra of the reciprocal plane}. Even more recently, in the hyperplane arrangements community discussions have started in regard to studying $\displaystyle\mathbb K\left[\ldots,\frac{1}{\prod_{i\in I}\ell_i},\ldots\right]$, where $I$ runs over all independent sets of $\A$, of certain given size. Due to \cite[Theorem 2.4]{GaSiTo}, the study of these new algebras can be done by analyzing the special fiber of certain ideals of $R$, generated by products of linear forms. Since for any $1\leq i<j\leq n$, the set $\{i,j\}$ is independent, this is the path we are pursuing for some parts of Section \ref{OT}, where we analyse some of the first properties of the algebra $$OT(2,\A):=\mathbb K\left[\ldots,\frac{1}{\ell_i\ell_j},\ldots\right], 1\leq i<j\leq n,$$ that we are calling {\em the Orlik-Terao algebra of second order of $\A$}.

Propositions \ref{connection} and \ref{properties} show that there is a strong connection between $OT(2,\A)$ and $OT(\A)$, yet despite that the generators of $I(\A)\subset\mathbb K[y_1,\ldots,y_n]$, the presentation ideal of $OT(\A)$, have nice combinatorial description (they are ``boundaries'' of circuits, see \cite{OrTe}), it is not the same for the generators of $I(2,\A)\subset\mathbb K[\ldots,t_{i,j},\ldots]$, the presentation ideal of $OT(2,\A)$. The difficulties occur due to the fact that the elements of $I(2,\A)$ are obtained from the elements of $I(\A)\cap \mathbb K[\ldots,y_iy_j,\ldots]$, and though in theory this looks reasonable, in practice the task to find them is challenging (see, e.g., Example \ref{example1}). By \cite[Theorem 2.4]{GaSiTo}, $OT(2,\A)$ is isomorphic to the special fiber of the ideal $I_{n-2}(\A)\subset R$, so by following the approach of \cite[Proposition 3.5]{GaSiTo}, one can obtain elements in $I(2,\A)$ from the generators of the symmetric ideal of $I_{n-2}(\A)$, via Sylvester forms. But even with this technique, as it is well known in the Ress algebras / elimination theory community, no one guarantees that one obtains all the generators of $I(2,\A)$. We conjecture that we do obtain all of them. None-the-less, our main result (Theorem \ref{theorem1}) helps determine all the generators of the symmetric ideal of $I_{n-2}(\A)$ (see Proposition \ref{gens_sym_ideal}).

\section{Ideals with linear free resolution}\label{linear}

Let $\Sigma=(\ell_1,\ldots,\ell_n)$ be a collection of linear forms in $R:=\mathbb K[x_1,\ldots,x_k]$, some of them, possibly proportional. Let $1\leq a\leq n$ be an integer and consider the ideal of $R$ $$I_a(\Sigma):=\langle \{\ell_{i_1}\cdots\ell_{i_a}|1\leq i_1<\cdots<i_a\leq n\}\rangle.$$

In this section we show that for the cases listed in the Introduction, $I_a(\Sigma)$ has linear graded free resolution. And we also look at some consequences of these results.

\subsection{The case $a=n-1$.} \label{a=n-1} Suppose $\Sigma=(\underbrace{\ell_1,\ldots,\ell_1}_{n_1},\ldots,\underbrace{\ell_s,\ldots,\ell_s}_{n_s})$, with $n_1,\ldots,n_s\geq 1$, and $\gcd(\ell_i,\ell_j)=1$, for any $1\leq i<j\leq s$. Suppose $|\Sigma|=n_1+\cdots+n_s=n$.

Let $\Sigma_0=\{\ell_1,\ldots,\ell_s\}$. It is immediate to observe that $$I_{n-1}(\Sigma)=(\ell_1^{n_1-1}\cdots\ell_s^{n_s-1})I_{s-1}(\Sigma_0).$$ By the proof of \cite[Lemma 3.1(a)]{GaSiTo}, we have the graded free resolution $$0\rightarrow R(-s)^{s-1}\rightarrow R(-(s-1))^s\rightarrow R\rightarrow R/I_{s-1}(\Sigma_0)\rightarrow 0.$$

By Hilbert-Burch Theorem (\cite[Theorem 20.15]{Ei2}), we obtain the linear graded free resolution

$$0\rightarrow R(-n)^{s-1}\rightarrow R(-(n-1))^s\rightarrow R\rightarrow R/I_{n-1}(\Sigma)\rightarrow 0.$$

\subsection{The case $k=2$.} \label{k=2} Suppose $\Sigma=(\ell_1,\ldots,\ell_n)\subset R=\mathbb K[x,y]$. Suppose that the linear form $\ell$ shows up at least twice in $\Sigma$. Let $\Sigma':=\Sigma\setminus\{\ell\}$, and $\Sigma'':=\Sigma'\setminus\{\ell\}$.

\begin{lem} \label{lemma_k=2} Let $1\leq a\leq n$. If
\begin{itemize}
  \item[(i)] $I_{a-1}(\Sigma')$ has linear graded free resolution, and
  \item[(ii)] $I_a(\Sigma'):\ell=I_{a-1}(\Sigma'')$,
\end{itemize} then $I_a(\Sigma):\ell=I_{a-1}(\Sigma')$ and $I_a(\Sigma)$ has linear graded free resolution.
\end{lem}
\begin{proof} First we show that $I_a(\Sigma):\ell=I_{a-1}(\Sigma')$.

Since we have $I_a(\Sigma)=\ell I_{a-1}(\Sigma')+I_a(\Sigma')$, the inclusion ``$\supseteq$'' is clear.

Let $f\in I_a(\Sigma):\ell$. Then $\ell f=\ell g+h$, for some $g\in I_{a-1}(\Sigma')$ and $h\in I_a(\Sigma')$. Therefore, $\ell(f-g)\in I_a(\Sigma')$, leading to $f-g\in I_a(\Sigma'):\ell$. From condition (ii), we get $f-g\in I_{a-1}(\Sigma'')$. But $\Sigma''\subset\Sigma'$, so $f\in I_{a-1}(\Sigma')$. This gives the other inclusion.

\medskip

Without any loss of generality, suppose $\ell=x$. Then for each $i=1,\ldots,n$, $\ell_i=b_ix+c_iy$, where $b_i,c_i\in \mathbb K$ some of them equal to zero. Suppose that precisely $1\leq s\leq n-2$ of the $c_i$'s are NOT equal to zero; if $s=0$, then $\Sigma=(\underbrace{x,\ldots,x}_{n})$, and part (6) in the Introduction shows that $I_a(\Sigma)$ has linear graded free resolution.

Let $\bar{\Sigma}:=(\underbrace{y,\ldots,y}_{s})$. Then $$\langle I_a(\Sigma),x\rangle=\langle I_a(\bar{\Sigma}),x\rangle.$$ If $a>s$, by convention, $I_a(\bar{\Sigma})=0$.

Everything put together gives the short exact sequence of $R$-modules

$$0\rightarrow R(-1)/I_{a-1}(\Sigma')\stackrel{\cdot x}\rightarrow R/I_a(\Sigma)\rightarrow R/\langle I_a(\bar{\Sigma}),x\rangle \rightarrow 0.$$

Condition (i), together with \cite[Theorem 1.2 (2)]{EiGo}, gives the Castelnuovo-Mumford regularity ${\rm reg}_R(R(-1)/I_{a-1}(\Sigma'))=a-2+1=a-1$.

If $a>s$, then ${\rm reg}_R(R/\langle I_a(\bar{\Sigma}),x\rangle)=0$, and if $a\leq s$, then ${\rm reg}_R(R/\langle I_a(\bar{\Sigma}),x\rangle)=a-1$.

But, from the inequalities of regularities under a short exact sequence (see \cite[Corollary 20.19 b.]{Ei})

$${\rm reg}(R/I_a(\Sigma))\leq \max\{{\rm reg}_R(R(-1)/I_{a-1}(\Sigma')),{\rm reg}_R(R/\langle I_a(\bar{\Sigma}),x\rangle)\}\leq a-1.$$

Since $I_a(\Sigma)$ is generated in degree $a$, we get that ${\rm reg}(R/I_a(\Sigma))=a-1$, and hence, from \cite[Theorem 1.2 (2)]{EiGo}, $I_a(\Sigma)$ has linear graded free resolution.
\end{proof}

\begin{thm}\label{theorem0} Let $\Sigma=(\ell_1,\ldots,\ell_n)\subset R=\mathbb K[x,y]$, be a collection of linear forms, some possibly proportional. Then, for any $1\leq a\leq n$, $I_a(\Sigma)$ has linear graded free resolution.
\end{thm}
\begin{proof} Let $\A$ be the reduced support of $\Sigma$; i.e., $\A$ consists of all nonproportional elements of $\Sigma$. If we show that $I_b(\A)$ has linear graded free resolution, and $I_b(\A):\ell=I_{b-1}(\A\setminus\{\ell\})$, for any $\ell\in\A$, and any $1\leq b\leq|\A|$, then via Lemma \ref{lemma_k=2} and the conventions in the Introduction, by adding one-by-one linear forms according to their multiplicity to obtain $\Sigma$, we obtain that $I_a(\Sigma)$ has linear graded resolution as well, for any $1\leq a\leq n$.

Suppose $\A=\{\ell_1,\ldots,\ell_m\}, m\leq n$, and $\gcd(\ell_i,\ell_j)=1$, for all $1\leq i < j\leq m$. If $m=1$, then part (6) of the introduction shows directly that $I_a(\Sigma)$ has linear graded free resolution.

Suppose $m\geq 2$. Then, ${\rm rank}(\A)=2$. Let $\mathcal C_{\A}$ be the linear code with generating matrix $G$ having columns dual to the linear forms of $\A$. This code has length $m$ and dimension 2. Since any two of the linear forms of $\A$ are linearly independent, the maximum number of columns of $G$ that span a $2-1=1$ dimensional vector space is 1. So the minimum distance of $\mathcal C_{\A}$ is $m-1$ (see, e.g., \cite[Remark 2.2]{ToVa}). But in these condition, by \cite[Theorem 3.1]{To}, we have indeed that for any $1\leq b\leq m-1$, $$I_b(\A)=\langle x,y\rangle^b.$$ Also $I_m(\A)=\langle \ell_1\cdots\ell_m\rangle$.

It is clear that for any $1\leq b\leq m$, $I_b(\A)$ has linear graded free resolution, and that $I_b(\A):\ell=I_{b-1}(\A\setminus\{\ell\})$, for any $\ell\in\A$.
\end{proof}

\subsection{The case $a=n-2$.} Let $\A=\{\ell_1,\ldots,\ell_n\}\subset R=\mathbb K[x_1,\ldots,x_k]$ be a hyperplane arrangement. Our main goal in this section is to show that the graded $R$-module $R/I_{n-2}(\A)$ has linear graded free resolution. For $1\leq i<j\leq n$, denote $$f_{i,j}:=\frac{\ell_1\ell_2\cdots\ell_n}{\ell_i\ell_j}\in R,$$ the generators of the ideal $I_{n-2}(\A)$.

Consider the complex of (graded) $R$-modules:

$${\bf CC}(\A):\, 0\rightarrow \bigoplus_{1\leq i<j\leq n}\frac{R(-(n-2))}{\langle \ell_i,\ell_j\rangle}\stackrel{\phi_{\A}} \longrightarrow\frac{R}{I_{n-1}(\A)}\stackrel{\pi_{\A}} \longrightarrow\frac{R}{I_{n-2}(\A)}\rightarrow 0,$$ where the map $\pi_{\A}$ is the natural surjection defined from the inclusion $I_{n-1}(\A)\subset I_{n-2}(\A)$ and the map $\phi_{\A}$ is defined as $$\phi_{\A}(\ldots,\widehat{h}_{i,j},\ldots)=\left[\sum_{1\leq i<j\leq n}h_{i,j}\left(\prod_{u\in[n]\setminus\{i,j\}}\ell_u\right)\right] \mbox{ mod }I_{n-1}(\A),$$ where $\displaystyle\widehat{h}_{i,j}\in \frac{R}{\langle \ell_i,\ell_j\rangle}$.

Obviously, $\displaystyle\prod_{u\in[n]\setminus\{i,j\}}\ell_u\in I_{n-2}(\A)$. This gives us that

$\bullet$ $\phi_{\A}$ is {\em well-defined}: if $h_{i,j},g_{i,j}\in R$ are such that $h_{i,j}-g_{i,j}\in\langle \ell_i,\ell_j\rangle$, then obviously $$(h_{i,j}-g_{i,j})\left(\prod_{u\in[n]\setminus\{i,j\}}\ell_j\right)\in I_{n-1}(\A).$$

$\bullet$ $\displaystyle {\rm Im}(\phi_{\A})=\frac{I_{n-2}(\A)}{I_{n-1}(\A)}={\rm ker}(\pi_{\A})$.

\medskip

A {\em 3-dependency} is a linear combination among \underline{exactly} three of the linear forms of $\A$. Suppose $${\rm D}_{i_1,i_2,i_3}:\,c_{i_1}\ell_{i_1}+c_{i_2}\ell_{i_2}+c_{i_3}\ell_{i_3}=0,\, 1\leq i_1<i_2<i_3\leq n,$$ is such a 3-dependency, where $c_{i_1},c_{i_2},c_{i_3}\in\mathbb K\setminus\{0\}$. In matroid language we say that $\{i_1,i_2,i_3\}$ is a {\em circuit}.

Let $\displaystyle F_{i_1,i_2,i_3}:=\frac{\ell_1\ell_2\cdots\ell_n}{\ell_{i_1}\ell_{i_2}\ell_{i_3}}\in R$. Then, by multiplying the dependency ${\rm D}_{i_1,i_2,i_3}$ with $F_{i_1,i_2,i_3}$, one obtains $$c_{i_1}f_{i_2,i_3}+c_{i_2}f_{i_1,i_3}+c_{i_3}f_{i_1,i_2}=0.$$

By denoting with $\Lambda(\A)$ to be the left-most $R$-module in the complex above, i.e.

$$\Lambda(\A):=\bigoplus_{1\leq i<j\leq n}\frac{R}{\langle \ell_i,\ell_j\rangle},$$ we just obtained that the element of $\Lambda(\A)$ $$\widehat{{\bf c}}_{i_1,i_2,i_3}:= (0,\ldots,0,\underbrace{\widehat{c}_{i_3}}_{(i_1,i_2)},0,\ldots,0,\underbrace{\widehat{c}_{i_2}}_{(i_1,i_3)},0,\ldots,0, \underbrace{\widehat{c}_{i_1}}_{(i_2,i_3)},0,\ldots,0)$$ is an element of the kernel of $\phi_{\A}$; the ``underbraces'' specify the position (or the summand) in the module $\Lambda(\A)$. In fact we have that the entire cyclic $R$-submodule $\mathcal R_{i_1,i_2,i_3}:=R\cdot \widehat{{\bf c}}_{i_1,i_2,i_3}$ is included in ${\rm ker}(\phi_{\A}),$ leading to $$\Lambda_3(\A)\subseteq {\rm ker}(\phi_{\A}),$$ where $\Lambda_3(\A)$ denotes the (internal) direct sum of the $R$-submodules $\mathcal R_{i_1,i_2,i_3}$ of $\Lambda(\A)$, running over all circuits $\{i_1,i_2,i_3\}$ of $\A$.

\begin{lem}\label{kernel} We have $\Lambda_3(\A)={\rm ker}(\phi_{\A})$.
\end{lem}

We are going to prove the lemma a bit later.

\medskip

Let $\displaystyle\mathcal V(\A):=\Lambda_3(\A)\otimes_{R}\mathbb K$ be the $\mathbb K$-vector subspace of $\mathbb K^m$, where $\displaystyle m:={{n}\choose{2}}$. Let $p(\A):=\dim_{\mathbb K}\mathcal V(\A)$, be the number of ``independent'' 3-dependencies of $\A$. In Claim 1 in the subsection following the proof of the main result, we prove that $\displaystyle p(\A)=\sum_{X\in L_2(\A)}{{|\A_X|-1}\choose{2}}$.

With these notations, our first main result of this paper is the following.

\begin{thm} \label{theorem1} Let $\A=\{\ell_1,\ldots,\ell_n\}\subset R$ be a hyperplane arrangement. Then the $R$-module $R/I_{n-2}(\A)$ has graded linear free resolution
$$0\rightarrow R^{m-n-p(\A)+1}(-n)\rightarrow R^{2m-n-2p(\A)}(-(n-1))\rightarrow R^{m-p(\A)}(-(n-2))\rightarrow R\rightarrow R/I_{n-2}(\A) \rightarrow 0.$$
\end{thm}
\begin{proof} From Lemma \ref{kernel} we have the acyclic (graded) complex of $R$-modules:

$$0\rightarrow \Lambda_3(\A)(-(n-2))\rightarrow \Lambda(\A)(-(n-2))\stackrel{\phi_{\A}} \longrightarrow\frac{R}{I_{n-1}(\A)}\stackrel{\pi_{\A}} \longrightarrow\frac{R}{I_{n-2}(\A)}\rightarrow 0.$$

Since each $R/\langle \ell_i,\ell_j\rangle, i\neq j$ is isomorphic as $R$-modules with $T:=\mathbb K[x_1,\ldots,x_{k-2}]$, if $k\geq 3$, or $T:=\mathbb K$, if $k=2$, then we have the isomorphism of $R$-modules $$\Lambda_3(\A)\simeq T^{p(\A)}.$$

Since $\langle \ell_i,\ell_j\rangle, i\neq j$ is a complete intersection, the $R$-module $T\simeq R/\langle \ell_i,\ell_j\rangle$ has the (Koszul) linear free resolution $$0\rightarrow R(-2)\rightarrow R^2(-1)\rightarrow R\rightarrow T\rightarrow 0.$$ This leads to the linear free resolutions:
$$0\rightarrow R^{p(\A)}(-n)\rightarrow R^{2p(\A)}(-(n-1))\rightarrow R^{p(\A)}(-(n-2))\rightarrow \Lambda_3(\A)(-(n-2))\rightarrow 0,$$ and
$$0\rightarrow R^m(-n)\rightarrow R^{2m}(-(n-1))\rightarrow R^m(-(n-2))\rightarrow \Lambda(\A)(-(n-2))\rightarrow 0.$$

Then, via {\em mapping cone}, the (graded) $R$-module ${\rm Im}(\phi_{\A})$, has a free resolution:
$$0\rightarrow R^{p(\A)}(-n)\rightarrow \begin{array}{c}
R^{2p(\A)}(-(n-1))\\
\oplus\\
R^m(-n)\\
\end{array} \rightarrow \begin{array}{c}
R^{p(\A)}(-(n-2))\\
\oplus\\
R^{2m}(-(n-1))\\
\end{array}\rightarrow R^m(-(n-2))\rightarrow {\rm Im}(\phi_{\A})\rightarrow 0.$$

By \cite[Lemma 1.13]{EiGo}, ${\rm Im}(\phi_{\A})$ has linear free resolution, and this can be obtained from the above resolution via appropriate ``cancelations'':

$$0\rightarrow R^{m-p(\A)}(-n)\rightarrow R^{2m-2p(\A)}(-(n-1))\rightarrow R^{m-p(\A)}(-(n-2))\rightarrow {\rm Im}(\phi_{\A})\rightarrow 0.$$

The $R$-module, $R/I_{n-1}(\A)$ has linear free resolution (see, e.g., \cite[Lemma 3.2]{Sc})

$$0\rightarrow R^{n-1}(-n)\rightarrow R^n(-(n-1))\rightarrow R\rightarrow R/I_{n-1}(\A)\rightarrow 0.$$

Then again by {\em mapping cone}, since ${\rm Im}(\phi_{\A})={\rm ker}(\pi_{\A})$, we obtain the free resolution

$$0\rightarrow R^{m-p(\A)}(-n)\rightarrow \begin{array}{c}
R^{2m-2p(\A)}(-(n-1))\\
\oplus\\
R^{n-1}(-n)\\
\end{array}\rightarrow \begin{array}{c}
R^{m-p(\A)}(-(n-2))\\
\oplus\\
R^n(-(n-1))\\
\end{array}\rightarrow R\rightarrow R/I_{n-2}(\A)\rightarrow 0.$$

Using again \cite[Lemma 1.13]{EiGo}, $R/I_{n-2}(\A)$ has a linear free resolution, and after the appropriate cancellations the linear free resolution is the one claimed in the statement. \end{proof}

\subsubsection{Proof of Lemma \ref{kernel}.} \label{the_proof} Suppose $X_1,\ldots, X_e$ are all the rank 2 flats in $L(\A)$, the lattice of intersections of $\A$, with $n_u:=|\A_{X_u}|\geq 3$ for all $u=1,\ldots,e$. For a flat $X\in L(\A)$, $\A_X$ denotes the subset of hyperplanes of $\A$ that contain $X$.

\medskip

\noindent$\bullet$ {\bf Claim 1:} One has $$\Lambda_3(\A)\cong\bigoplus_{u=1}^e\Lambda_3(\A_{X_u})\mbox{ and } p(\A)=\dim_{\mathbb K}\mathcal V(\A)=\sum_{u=1}^e{{n_u-1}\choose{2}}.$$

{\em Proof Claim 1}. If $\{i_1,i_2,i_3\}$ is a circuit, then $X:=V(\ell_{i_1},\ell_{i_2},\ell_{i_3})$ is a rank 2 flat with $|\A_X|\geq 3$. If a non-zero tuple belongs to $\Lambda_3(\A_{X_u})\cap\Lambda_3(\A_{X_v}),u\neq v$, then it has at least one non-zero entry, say the $(i,j)$-th entry, that gives that there exist a circuit $\{i,j,a\}$ of $\A_{X_u}$, and a circuit $\{i,j,b\}$ of $\A_{X_v}$. Since $X_u=V(\ell_i,\ell_j)$ and $X_v=V(\ell_i,\ell_j)$, we get a contradiction.

For the second part we have to show that if $X$ is a rank 2 flat with $s:=|\A_X|\geq 3$, then $\displaystyle p(\A_X)={{s-1}\choose{2}}$. After a change of variables and some reordering of the hyperplanes in $\A$, we can suppose $\A_X=\{\ell_1,\ldots,\ell_s\}\subset \mathbb K[x_1,x_2]\subset R$, with $\ell_1=x_1$, and $\ell_i=x_1+\lambda_ix_2,\lambda_i\in\mathbb K\setminus\{0\},\, i=2,\ldots,s$, and $\lambda_i\neq\lambda_j$, if $i\neq j$.

Any three of the linear forms of $\A_X$ lead to a dependency, and for any $2\leq u<v<w\leq s$ we have

\begin{eqnarray}
{\rm D}_{u,v,w}:&& (\lambda_v-\lambda_w)\ell_u+(\lambda_w-\lambda_u)\ell_v+(\lambda_u-\lambda_v)\ell_w=0\nonumber\\
{\rm D}_{1,u,v}:&& (\lambda_u-\lambda_v)\ell_1+\lambda_v\ell_u+(-\lambda_u)\ell_v=0\nonumber\\
{\rm D}_{1,u,w}:&& (\lambda_u-\lambda_w)\ell_1+\lambda_w\ell_u+(-\lambda_u)\ell_w=0.\nonumber
\end{eqnarray} It is easy to check that $${\rm D}_{u,v,w}=\frac{1}{\lambda_u}[(\lambda_u-\lambda_w){\rm D}_{1,u,v}-(\lambda_u-\lambda_v) {\rm D}_{1,u,w}],$$ which leads to $${\bf c}_{u,v,w}=\frac{1}{\lambda_u}[(\lambda_u-\lambda_w){\bf c}_{1,u,v}-(\lambda_u-\lambda_v) {\bf c}_{1,u,w}].$$ This means that $\mathcal V(\A_X)$ is generated by the vectors ${\bf c}_{1,u,v}$, where $2\leq u<v\leq s$.

Suppose there exist $\gamma_{u,v}\in\mathbb K, 2\leq u<v\leq s$, such that $\displaystyle {\bf c}:=\sum_{2\leq u<v\leq s}\gamma_{u,v}{\bf c}_{1,u,v}=0$. For each pair $2\leq u<v\leq s$, the $(u,v)$-entry of ${\bf c}$ equals to $\gamma_{u,v} (\lambda_u-\lambda_v)$, which must be zero. Since $\lambda_u\neq\lambda_v$, we get that $\gamma_{u,v}=0$.

So $\{{\bf c}_{1,u,v}|2\leq u<v\leq s\}$ is a basis for $\mathcal V(\A_X)$, and hence the claim.

\medskip

\noindent$\bullet$ {\bf Claim 2:} Lemma \ref{kernel} is true if $n=2$.

{\em Proof Claim 2}. We have $\A=\{\ell_1,\ell_2\}\subset R$, $\gcd(\ell_1,\ell_2)=1$. Obviously, $\Lambda_3(\A)=0$. With the convention that $I_0(\A)=R$, and since $I_1(\A)=\langle \ell_1,\ell_2\rangle$, then the complex ${\bf CC}(\A)$ translates into the complex

$$0\rightarrow \frac{R}{\langle \ell_1,\ell_2\rangle}\stackrel{\phi_{\A}} \longrightarrow\frac{R}{\langle \ell_1,\ell_2\rangle}\rightarrow 0,$$ where $\phi_{\A}$ is the identity map, and hence ${\rm ker}(\phi_{\A})$ is also zero.

\medskip

\noindent$\bullet$ {\bf Claim 3:} Lemma \ref{kernel} is true if $k=2$.

{\em Proof Claim 3}. Let $\A=\{\ell_1,\ldots,\ell_n\}\subset R:=\mathbb K[x_1,x_2]$ with $\gcd(\ell_i,\ell_j)=1$, if $i\neq j$. From Claim 2, suppose $n\geq 3$. Let us look at the map $\phi_{\A}$:
$$\bigoplus_{1\leq i<j\leq n}\frac{R(-(n-2))}{\langle \ell_i,\ell_j\rangle}\stackrel{\phi_{\A}} \longrightarrow\frac{R}{I_{n-1}(\A)},$$ with $$\phi_{\A}(\ldots,\widehat{h}_{i,j},\ldots)=\left[\sum_{1\leq i<j\leq n}h_{i,j}\left(\prod_{u\in[n]\setminus\{i,j\}}\ell_u\right)\right] \mbox{ mod }I_{n-1}(\A).$$

Since for any $i\neq j$, $\langle\ell_i,\ell_j\rangle={\frak m}=\langle x_1,x_2\rangle$, if degree of $h_{i,j}$ is $\geq 1$, then $\widehat{h}_{i,j}=0$ in $R/\langle\ell_i,\ell_j\rangle$. Then $\displaystyle{\rm ker}(\phi_{\A})\subset\mathbb K^{{n}\choose{2}}$ is the $\mathbb K$-vector subspace of all $\mathbb K$-dependencies among the standard generators $f_{i,j}$ of $I_{n-2}(\A)$. So
$$\dim_{\mathbb K}{\rm ker}(\phi_{\A})={{n}\choose{2}}-\mu(I_{n-2}(\A)),$$ where $\mu(I_{n-2}(\A))$ is the minimum number of generators of $I_{n-2}(\A)$.

Same coding theory argument as in the proof of Theorem \ref{theorem0} gives $$I_{n-2}(\A)=\langle x_1,x_2\rangle^{n-2}.$$ Therefore, $\mu(I_{n-2}(\A))=n-1$, leading to $$\dim_{\mathbb K}{\rm ker}(\phi_{\A})={{n}\choose{2}}-(n-1)={{n-1}\choose{2}}.$$ But from the proof of the second part of Claim 1, this is exactly $p(\A)$, leading to $\Lambda_3(\A)={\rm ker}(\phi_{\A})$.

\vskip .5in

At this moment we proceed to prove Lemma \ref{kernel}. Let $\A=\{\ell_1,\ldots,\ell_n\}\subset R:=\mathbb K[x_1,\ldots,x_k], k\geq 2$ with $\gcd(\ell_i,\ell_j)=1$, if $i\neq j$.

We will use induction on $|\A|=n\geq 2$. From Claim 2, the base case $n=2$ is verified.

Suppose $n\geq 3$. For $1\leq i<j\leq n$, let $h_{i,j}\in R$, such that $\phi_{\A}(\ldots,\widehat{h}_{i,j},\ldots)=0$ in $R/I_{n-1}(\A)$. So
$${\bf P}:=\sum_{1\leq i<j\leq n}h_{i,j}\left(\prod_{u\in[n]\setminus\{i,j\}}\ell_u\right)\in I_{n-1}(\A).$$

If ${\rm rank}(\A)=2$, then modulo a change of variables, Claim 3 proves the result. So assume that ${\rm rank}(\A)\geq 3$.

For any $1\leq i<j\leq n$, there exists $u(i,j)\in[n]\setminus\{i,j\}$ such that $$\ell_{u(i,j)}\notin\langle \ell_i,\ell_j\rangle.$$ This is true because, otherwise there would exist $1\leq i_0<j_0\leq n$ such that for all $v=1,\ldots,n$, we would have $\ell_v\in \langle\ell_{i_0},\ell_{j_0}\rangle$, and so ${\rm rank}(\A)=2$; a contradiction with the assumption we made above.

Without any loss of generality, suppose $h_{1,2}\neq 0$, and suppose $u(1,2)=n$. Let $\A':=\A\setminus\{\ell_n\}$, and denote $n':=n-1=|\A'|$.

We can rewrite

$${\bf P}=\sum_{1\leq i<j\leq n-1}\left[\left(\ell_nh_{i,j}+\frac{1}{2}h_{i,n}\ell_j+\frac{1}{2}h_{j,n}\ell_i\right) \prod_{v\in[n-1]\setminus\{i,j\}}\ell_v\right].$$ Since $I_{n-1}(\A)=\ell_n I_{n'-1}(\A')+I_{n'}(\A')\subset I_{n'-1}(\A')$, we have that $$\phi_{\A'}(\ldots,\ell_n\widehat{h}_{i,j},\ldots)=0,$$ where the argument of the map has $\displaystyle {{n-1}\choose{2}}$ entries, as $1\leq i<j\leq n-1$.

\medskip

If the rank 2 flat $X:=V(\ell_1,\ell_2)$ has $|\A_X|=2$ (i.e., there is no 3-dependency of $\A$ containing both $\ell_1$ and $\ell_2$), then also $|\A'_X|=2$. By induction, ${\rm ker}(\phi_{\A'})=\Lambda_3(\A')$, so  $\ell_n\widehat{h}_{1,2}=0$ in $R/\langle\ell_1,\ell_2\rangle$. Since $\ell_n\notin\langle\ell_1,\ell_2\rangle$, we obtain $\widehat{h}_{1,2}=0$ in $R/\langle\ell_1,\ell_2\rangle$.

\medskip

Suppose the flat $X=V(\ell_1,\ell_2)$ from above has $\A_X=\{\ell_1,\ell_2,\ldots,\ell_s\}\subset\A$ with $s\geq 3$. Since $\ell_n\notin\langle\ell_1,\ell_2\rangle$, we have again that $\A'_X=\A_X\subseteq\A'$.\footnote{It is also clear, that for any $a\neq b\in\{1,\ldots,s\}$, we can pick $u(a,b)=n$.}

By induction, together with Claim 1, we have that the $\displaystyle {{s}\choose{2}}$-tuple

$$(\ell_n\widehat{h}_{1,2},\ldots,\ell_n\widehat{h}_{i,j},\ldots,\ell_n\widehat{h}_{s-1,s})\in\Lambda_3(\A'_X).$$

Same as in the proof of Claim 1, after a change of variables, let us assume that $\ell_1=x_1$, and $\ell_i=x_1+\lambda_ix_2,\lambda_i\in\mathbb K\setminus\{0\},\, i=2,\ldots,s$, and $\lambda_i\neq\lambda_j$, if $i\neq j$. Then

$$(\ell_n\widehat{h}_{1,2},\ldots,\ell_n\widehat{h}_{i,j},\ldots,\ell_n\widehat{h}_{s-1,s})=\sum_{2\leq u<v\leq s}\widehat{g}_{u,v}{\bf c}_{1,u,v},$$ for some $\widehat{g}_{u,v}\in R/\langle \ell_1,\ell_2\rangle$\footnote{For any $u\neq v\in\{1,\ldots,s\}$, we have $\langle\ell_u,\ell_v\rangle=\langle\ell_1,\ell_2\rangle=\langle x_1,x_2\rangle$.}, and where for $2\leq u<v\leq s$,

$${\bf c}_{1,u,v}=(0,\ldots,0,\underbrace{-\lambda_u}_{(1,u)},0,\ldots,0,\underbrace{\lambda_v}_{(1,v)},0,\ldots,0, \underbrace{(\lambda_u-\lambda_v)}_{(u,v)},0,\ldots,0).$$

Equating each entry, we obtain that in $R/\langle \ell_1,\ell_2\rangle$ we have
\begin{enumerate}
  \item For each $2\leq u<v\leq s$, $$\ell_n\widehat{h}_{u,v}=(\lambda_u-\lambda_v)\widehat{g}_{u,v}.$$
  \item For each $2\leq w\leq s$, $$\ell_n\widehat{h}_{1,w}=\lambda_w\left(\sum_{t=2}^{w-1}\widehat{g}_{t,w}-\sum_{t=w+1}^s\widehat{g}_{w,t}\right).$$
\end{enumerate}

Solving for $\widehat{g}_{u,v}$ in the first group of equations, and plugging in the second group of equations, we obtain that for each $2\leq w\leq s$

$$\ell_n\widehat{h}_{1,w}=\ell_n\lambda_w\left(\sum_{t=2}^{w-1}\frac{1}{\lambda_t-\lambda_w}\widehat{h}_{t,w}- \sum_{t=w+1}^s\frac{1}{\lambda_w-\lambda_t}\widehat{h}_{w,t}\right),$$ which we can immediately simplify by $\ell_n$, since $\ell_n\notin\langle\ell_1,\ell_2\rangle$.

This lead to the following equation:

$$(\widehat{h}_{1,2},\ldots,\widehat{h}_{i,j},\ldots,\widehat{h}_{s-1,s})=\sum_{2\leq u<v\leq s}\frac{\widehat{h}_{u,v}}{\lambda_u-\lambda_v}{\bf c}_{1,u,v}.$$ One should observe that $\ell_n$ does not show up anywhere in this equation. What show up are all the pairs of distinct indices of the linear forms of $\A_X$, where $X=V(\ell_u,\ell_v)$, for any $u\neq v\in\{1,\ldots,s\}$.

\medskip

The two cases above mean that ${\rm ker}(\phi_{\A})$ is isomorphic to the direct sum of $\Lambda_3(\A_X)$, where $X$ scans over all rank 2 flats of $\A$ with $|\A_X|\geq 3$. With Claim 1, this proves the desired equality $${\rm ker}(\phi_{\A})=\Lambda_3(\A).$$

\vskip .5in

Before concluding this subsection, it is worth observing the following facts. If $X$ is a rank 2 flat with $|\A_X|=2$, then obviously $\Lambda_3(\A_X)=0$. So we showed that $\displaystyle {\rm ker}(\phi_{\A})\cong\bigoplus_{X\in L_2(\A)}{\rm ker}(\phi_{\A_X})$, and from Claims 2 and 3, each direct summand is isomorphic to $\Lambda_3(\A_X)$.

In fact we obtained that the $R$-module ${\rm ker}(\phi_{\A})$ ``decomposes by localizations''. The prime ideals ${\frak p}_{i,j}:=\langle\ell_i,\ell_j\rangle$ are the associated primes of the $R$-module $\Lambda(\A)$. By localization we have two cases:

\begin{enumerate}
  \item If ${\frak p}_{i,j}\neq {\frak p}_{u,v}$, then ${\frak p}_{u,v}R_{{\frak p}_{i,j}}=R_{{\frak p}_{i,j}}$, and therefore $$\left(\frac{R}{{\frak p}_{u,v}}\right)_{{\frak p}_{i,j}}=0.$$
  \item If ${\frak p}_{i,j}={\frak p}_{u,v}$, then ${\frak p}_{u,v}R_{{\frak p}_{i,j}}={\frak p}_{i,j}R_{{\frak p}_{i,j}}$, and therefore $$\left(\frac{R}{{\frak p}_{u,v}}\right)_{{\frak p}_{i,j}}\simeq\mathbb K,$$ as $R$-modules.
\end{enumerate}

From the proof above, it becomes transparent that by localization at a prime ideal defining a rank 2 flat $X:=V({\frak p}_{i,j})$, one obtains ${\rm ker}(\phi_{\A_X})$, and therefore the decomposition $${\rm ker}(\phi_{\A})\simeq \bigoplus_{{\frak p}\in Ass_R(\Lambda(\A))}{\rm ker}(\phi_{\A})_{{\frak p}}.$$

\vskip .5in

If for any hyperplane arrangement $\A\subset R$, the $R$-module $R/I_{n-1}(\A)$ is (arithmetically) Cohen-Macaulay (see \cite[Lemma 3.2]{Sc}), by comparison, that is the case for $R/I_{n-2}(\A)$ only in special cases.

\begin{cor}\label{CM} Let $\A\subset R$ be a hyperplane arrangement with $|\A|=n\geq 2$. Then $R/I_{n-2}(\A)$ is (arithmetically) Cohen-Macaulay if and only if either
\begin{itemize}
  \item[(1)] if ${\rm rank}(\A)=2$, then $\displaystyle p(\A)={{n-1}\choose{2}}$, or
  \item[(2)] if ${\rm rank}(\A)\geq 3$, then $p(\A)=0$ (i.e., $\A$ is 3-generic, or any three linear forms of $\A$ are linearly independent).
\end{itemize}
\end{cor}
\begin{proof} From the minimal free resolution exhibited in Theorem \ref{theorem1}, the projective dimension is
$${\rm pdim}_R(R/I_{n-2}(\A))=\left\{
  \begin{array}{ll}
    2, & \hbox{if $p(\A)={{n-1}\choose{2}}$;} \\
    3, & \hbox{otherwise.}
  \end{array}
\right.$$

From proof of Claim 3, if ${\rm rank}(\A)=2$, then $\displaystyle p(\A)={{n-1}\choose{2}}$. For the converse, let $X_1,\ldots,X_r$ be all the distinct rank 2 flats of $L(\A)$, and suppose $|\A_{X_u}|=n_u$, for $u=1,\ldots,r$. Obviously, $n_u\geq 2$, and from Claim 1
$$\sum_{u=1}^r{{n_u-1}\choose{2}}={{n-1}\choose{2}}.$$ Also, for any $1\leq i<j\leq n$, we have that $V(\ell_i,\ell_j)$ is a rank 2 flat of $\A$, and therefore
$$\sum_{u=1}^r{{n_u}\choose{2}}={{n}\choose{2}}.$$ Subtracting these two equations one obtains
$$n_1+\cdots+n_r-r=n-1.$$ Modulo the Claim 4 below, this is true only when ${\rm rank}(\A)=2$ (in this case $r=1$ and $n_1=n$).

\medskip

\noindent$\bullet${\bf Claim 4:} If ${\rm rank}(\A)\geq 3$, then $n_1+\cdots+n_r-r\geq n$.

{\em Proof of Claim 4.} We will use induction on $|\A|=n\geq 3$. The base case, $n=3$, is immediate, since $\A$ will consist of three linearly independent linear forms, and so $r=3$, and $n_1=n_2=n_3=2$; indeed, $2+2+2-3=3\geq 3$.

For the induction step, suppose $|\A|\geq 4$, and let $\A':=\A\setminus\{\ell_n\}$. Suppose that for some $1\leq b\leq r$, $X_i\subset V(\ell_n)$ for $1\leq i\leq b$, and $X_j\nsubseteq V(\ell_n)$ for $b+1\leq j\leq r$. Also suppose for some $0\leq a\leq b$, $n_u=2$ for $0\leq u\leq a$, and $n_v\geq 3$ for $a+1\leq v\leq b$.

If ${\rm rank}(\A')\geq 3$, then by induction, the claimed inequality is true for $\A'$, so we have

$$(n_{a+1}-1)+\cdots+(n_b-1)+n_{b+1}+\cdots+n_r-(r-a)\geq n-1,$$ which give $$2a+n_{a+1}+\cdots+n_r-r\geq n-1+b.$$ Since $n_1+\cdots+n_a=2a$, and $b\geq 1$, the method of induction proves the inequality.

If ${\rm rank}(\A')=2$, then, since ${\rm rank}(\A)=3$, after a change of variables we have that in $\mathbb P^2$, $\A'$ is a pencil of $n-1$ lines through a point, and $\ell_n$ is a line that misses that point. So $r=n$, and $n_1=n-1$, and $n_i=2$ for $2\leq i\leq n$. With this, we have $$n_1+\cdots+n_r-r=(n-1)+2(n-1)-n=2n-3\geq n.$$

\medskip

From the beginning of \cite[Section 2]{To}, any minimal prime of $I_{n-2}(\A)$ is of the form $\langle \ell_u,\ell_v,\ell_w\rangle$, for some (any) $1\leq u<v<w\leq n$. So ${\rm ht}(I_{n-2}(\A))=3$, if and only if $\Lambda_3(\A)=0$. But this is equivalent to $p(\A)=0$.
\end{proof}

\subsubsection{The case of three variables, $k=3$.} \label{3vars} Let $\A=\{\ell_1,\ldots,\ell_n\}\subset R:=\mathbb K[x,y,z]$ be a line arrangement in $\mathbb P^2$, and suppose ${\rm rank}(\A)=3$. Let ${\frak m}:=\langle x,y,z\rangle$. Then, all the flats of rank 2 in $L(\A)$ correspond to the points of intersection among the lines of $\A$, say $P_1,\ldots, P_s$ (distinct). Often, this set is denoted $Sing(\A)$, and it is called {\em the singularity locus of $\A$}.

For $i=1,\ldots,s$, let $n_i:=|\A_{P_i}|$ be the number of lines of $\A$ intersecting at $P_i$.

Lemmas 3.1 and 3.2 in \cite{Sc} say that $I_{n-1}(\A)$ has primary decomposition $$I_{n-1}(\A)=I(P_1)^{n_1-1}\cap\cdots\cap I(P_s)^{n_s-1}.$$
In this section we are interested in finding a similar primary decomposition, but for $I_{n-2}(\A)$.

\medskip

First suppose that $\A$ is generic (i.e., any three of the linear forms of $\A$ are linearly independent). Let $\mathcal C_{\A}$ be the linear code with generating matrix $G$ having columns dual to the linear forms defining $\A$. This code has length $n$ and dimension 3. Since any three of the linear forms of $\A$ are linearly independent, the maximum number of columns of $G$ that span a $3-1=2$ dimensional vector space is 2. So the minimum distance of $\mathcal C_{\A}$ is $n-2$ (see, e.g., \cite[Remark 2.2]{ToVa}). But in these condition, by \cite[Theorem 3.1]{To}, we have that $I_{n-2}(\A)={\frak m}^{n-2}$, which is primary.

\medskip

Suppose $\A$ is not generic. Then ${\rm ht}(I_{n-2}(\A))=2$. Since ${\rm ht}(I_{n-1}(\A))=2$, and ${\rm ht}(I_n(\A))=1$, by Propositions 2.2 and 2.3 in \cite{AnGaTo}, we have $$I_{n-2}(\A)=I(P_1)^{n_1-2}\cap\cdots\cap I(P_s)^{n_s-2}\cap K,$$ where $K$ is ${\frak m}$-primary ideal. Also, for any ideal $I\subset R$, by convention $I^0=R$.

Let $I$ be an ideal of $R$. Then the {\em saturation of $I$} is the ideal $$I^{sat}:=\{f\in R|{\frak m}^{v_f}\cdot f\in I,\mbox{ for some integer }v_f\geq 0\}.$$ Obviously, $I\subseteq I^{sat}$.

Since $\sqrt{K}={\frak m}$, there exists an integer $w\geq 1$, such that ${\frak m}^w\subseteq K$. So $K^{sat}=R$. For any $j=1,\ldots, s$, $I(P_j)$ is a linear prime ideal of height 2. So there exists one generator $L_j$ of ${\frak m}$ that does not belong to $I(P_j)$. But then, $I(P_j)^{n_j-2}:L_j=I(P_j)^{n_j-2}$, giving that $(I(P_j)^{n_j-2})^{sat}=I(P_j)^{n_j-2}$. These, together with the formula $(I\cap J)^{sat}=I^{sat}\cap J^{sat}$, give

$$I_{n-2}(\A)^{sat}=I(P_1)^{n_1-2}\cap\cdots\cap I(P_s)^{n_s-2}.$$

For any ideal $I$ of $R$, by definition, $${\rm H}_{\frak m}^0(R/I)=\{\hat{f}\in R/I|{\frak m}^{v_f}\cdot f\in I\},$$ so ${\rm H}_{\frak m}^0(R/I)=I^{sat}/I$.

Denote, $M:={\rm H}_{\frak m}^0(R/I_{n-2}(\A))$. By Theorem \ref{theorem1}, ${\rm reg}(R/I_{n-2}(\A))=n-3$. Therefore, by Corollaries 4.5 and 4.4 in \cite{Ei}, $\max\{d|M_d\neq 0\}\leq n-3$, which leads to $$(I_{n-2}(\A)^{sat}/I_{n-2}(\A))_e=0, \mbox{ for any }e\geq n-2.$$ This means that $$I_{n-2}(\A)=I_{n-2}(\A)^{sat}\cap{\frak m}^{n-2}.$$

Considering that if $\A$ is generic, then $n_j=2$ for all $j=1,\ldots,s$, putting everything together we have the following result

\begin{prop}\label{primaryDecomp} Let $\A$ be an essential line arrangement in $\mathbb P^2$. Suppose $|\A|=n$, and $Sing(\A)=\{P_1,\ldots,P_s\}$, with $n_j=|\A_{P_j}|, j=1,\ldots,s$. Then, in $R=\mathbb K[x,y,z]$, we have the primary decomposition

$$I_{n-2}(\A)=I(P_1)^{n_1-2}\cap\cdots\cap I(P_s)^{n_s-2}\cap \langle x,y,z\rangle^{n-2}.$$
\end{prop}

\section{Orlik-Terao algebra of the second order}\label{OT}

Let $\A=\{\ell_1,\ldots,\ell_n\}\subset R:=\mathbb K[x_1,\ldots,x_k]$. According to \cite[Example 2.2(iii)]{GaSiTo}, for $S_{i,j}:=[n]\setminus\{i,j\}, 1\leq i<j\leq n$, and $\mathfrak S:=\{\ldots,S_{i,j},\ldots\}$, we define {\em the Orlik-Terao algebra of the second order of $\A$} to be $$OT(2,\A):=OT(\mathfrak S,\A)=\mathbb K\left[\ldots,\frac{1}{\ell_i\ell_j},\ldots\right].$$

In this section we study the first properties of this algebra, also making some links, if they exist, with $OT(\A)$, the (classical) Orlik-Terao algebra of $\A$.

From \cite[Proposition 2.3 and Theorem 2.4]{GaSiTo}, with the notations at the beginning of Section 2, i.e., $f:=\ell_1\cdots\ell_n$, and $f_{i,j}:=f/(\ell_i\ell_j), 1\leq i<j\leq n$, we have the following isomorphisms of graded $\mathbb K-$algebra

$$OT(2,\A)\simeq\mathbb K[\ldots,f_{i,j},\ldots]\simeq\mathcal F(I_{n-2}(\A)),$$ where $\mathcal F(I_{n-2}(\A))$ is the special fiber of the ideal $I_{n-2}(\A)$.

Denote ${\bf T}:=\mathbb K[\ldots,t_{i,j},\ldots], 1\leq i<j\leq n$. Because of the above isomorphism with the special fiber, the defining ideal of $OT(2,\A)$ is $$I(2,\A):=\{F\in {\bf T}| F(\ldots,f_{i,j},\ldots)=0\},$$ and so $OT(2,\A)\simeq {\bf T}/I(2,\A)$.

\begin{prop} \label{dimension} If ${\rm rank}(\A)=k\geq 2$ (i.e., $\A$ is essential), then the Krull dimension of $OT(2,\A)$, and therefore the analytic spread of $I_{n-2}(\A)$, equals $k$.
\end{prop}
\begin{proof} We have $\displaystyle OT(2,\A)=\mathbb K\left[\ldots,\frac{1}{\ell_i\ell_j},\ldots\right]$, which is an integral domain. Then the total field of fractions is $Q(OT(2,\A))=\mathbb K(\ldots,\ell_i\ell_j,\ldots)$.

Since $\A$ is essential, after a change of variables we can suppose that $\ell_i=x_i$, for $i=1,\ldots,k$. We have the following sequence of inclusions:

$$\mathbb K(\ldots,x_ix_j,\ldots)\subseteq Q(OT(2,\A))\subset \mathbb K(x_1,\ldots,x_k),$$ where the left-most field is the field of fraction of $\mathbb K[G]$, where $G$ is the complete graph on $k$ vertices. By \cite[Corollary 10.1.21]{Vi}, the Krull dimension of $\mathbb K[G]$ is $k$, and therefore ${\rm tr.deg}_{\mathbb K}\mathbb K(\ldots,x_ix_j,\ldots)=k$. Since ${\rm tr.deg}_{\mathbb K}\mathbb K(x_1,\ldots,x_k)=k$, we obtain $${\rm tr.deg}_{\mathbb K}Q(OT(2,\A))=k,$$ and hence, the result.
\end{proof}

Since $I_{n-2}(\A)$ is linearly presented, an immediate consequence of Proposition \ref{dimension}, via \cite[Theorem 3.2]{DoHaSi}, and under the assumption that $\A$ is essential, is that the rational map $$\mathbb P^{k-1}\dashrightarrow \mathbb P^{m-1}, [x_1,\ldots,x_k]\mapsto [\ldots,1/(\ell_i\ell_j),\ldots], 1\leq i<j\leq n,$$ where $\displaystyle m={{n}\choose{2}}$, is birational onto its image.

\vskip .5in

For the remainder of this article, we focus our attention on the generators of $I(2,\A)$. First, we have the following lemma.

\begin{lem}\label{standard_generators} The generators of $I(2,\A)$ include the following standard elements:
\begin{enumerate}
  \item LINEAR: If $c_{i_1}\ell_{i_1}+c_{i_2}\ell_{i_2}+c_{i_3}\ell_{i_3}=0$, $1\leq i_1<i_2<i_3\leq n$, is a 3-dependency, then $$\underbrace{c_{i_1}t_{i_2,i_3}+c_{i_2}t_{i_1,i_3}+c_{i_3}t_{i_1,i_2}}_{L_{i_1,i_2,i_3}}\in I(2,\A).$$
  \item QUADRATIC: If $n\geq 4$, then for any $1\leq u<v<w<z\leq n$ $$\underbrace{t_{u,v}t_{w,z}-t_{u,w}t_{v,z}}_{Q_{u,v,w,z}^1},\underbrace{t_{u,v}t_{w,z}-t_{u,z}t_{v,w}}_{Q_{u,v,w,z}^2}\in I(2,\A).$$
\end{enumerate}
\end{lem}
\begin{proof} The linear generators are obtained from multiplying the dependency by $\displaystyle \prod_{j\in[n]\setminus\{i_1,i_2,i_3\}}\ell_j$, obtaining $c_{i_1}f_{i_2,i_3}+c_{i_2}f_{i_1,i_3}+c_{i_3}f_{i_1,i_2}=0$.

The quadratic generators are obtained from the commutativity of the products of linear forms.
\end{proof}

Denote ${\bf S}:=\mathbb K[y_1,\ldots,y_n]$, and denote $I(\A)$ to be the defining ideal of the (usual) Orlik-Terao algebra, i.e., $OT(\A)\simeq {\bf S}/I(\A)$.

We have the following first result.

\begin{prop}\label{connection} The map $$OT(2,\A)\longrightarrow OT(\A),\, t_{i,j}\mapsto y_iy_j$$ is a well-defined embedding of algebras.
\end{prop}
\begin{proof} First we show ``well-defined''. Let $F\in I(2,\A)$ be homogeneous in the $t_{i,j}$ variables, of degree $d$. Then
$$F(\ldots,f_{i,j},\ldots)=F\left(\ldots,\frac{f}{\ell_i\ell_j},\ldots\right)=0.$$ Multiplying this by $f^d$, and distributing $f$ with the appropriate powers, we obtain $$F\left(\ldots,\frac{f}{\ell_i}\cdot\frac{f}{\ell_j},\ldots\right)=0,$$ which means that $F(\ldots,y_iy_j,\ldots)\in I(\A)\subset{\bf S}$.

The ``embedding'' part of the statement follows the reverse argument: let $F\in {\bf T}$ be homogeneous in the $t_{i,j}$ variables, of degree $d$, such that $F(\ldots,y_i y_j,\ldots)\in I(\A)$. Then, $$F\left(\ldots,\frac{f}{\ell_i}\cdot\frac{f}{\ell_j},\ldots\right)=0.$$ Taking $f^d$ common factor from all the terms, after regrouping we obtain $$F\left(\ldots,\frac{f}{\ell_i\ell_j},\ldots\right)=F(\ldots,f_{i,j},\ldots)=0,$$ hence $F\in I(2,\A)$.
\end{proof}

Observe that the linear elements obtained in Lemma \ref{standard_generators}, via the map in Proposition \ref{connection}, give the well-known quadratic elements $c_{i_1}y_{i_2}y_{i_3}+c_{i_2}y_{i_1}y_{i_3}+c_{i_3}y_{i_1}y_{i_2}$ of $I(\A)$. The quadratic generators obtained there give only that $0\in I(\A)$.

\medskip

Let $J(\A)$ be the ideal of ${\bf S}$, defined as: $$J(\A):=\{F(\ldots,y_iy_j,\ldots)|F\in I(2,\A)\subset{\bf T}\}.$$ From Proposition \ref{connection} we have $J(\A)\subseteq I(\A)$.

Consider the subring ${\bf S'}:=\mathbb K[\ldots,y_iy_j,\ldots]\subset {\bf S}$, and let $J'(\A):=J(\A)\cap {\bf S'}$. Naturally, ${\bf S'}$ is isomorphic as a graded $\mathbb K$-algebra to ${\bf T}$ quotient by the standard quadratic generators presented in Lemma \ref{standard_generators}.

\begin{prop} \label{properties} We have the following properties:
\begin{enumerate}
  \item In ${\bf S}$, one has $J(\A):\langle y_1,\ldots,y_n\rangle=I(\A)$.
  \item $J'(\A)=I(\A)\cap {\bf S'}$.
  \item $\displaystyle OT(2,\A)\simeq \frac{{\bf S'}}{J'(\A)}$.
\end{enumerate}
\end{prop}
\begin{proof} To prove (1), suppose $c_{i_1}\ell_{i_1}+\cdots+c_{i_s}\ell_{i_s}=0$ is a minimal dependency in $\A$. This leads to the generator $$G_{i_1,\ldots,i_s}:=c_{i_1}y_{i_2}\cdots y_{i_s}+\cdots+ c_{i_s}y_{i_1}\cdots y_{i_{s-1}}$$ of $I(\A)$ (see \cite{OrTe}). We have two cases:

\begin{itemize}
  \item If $s=2a+1$, then $s-1$ is even, and in each term we can group pairs of $y$'s with different indices to obtain an element of $I(2,\A)$ of degree $a=(s-1)/2$.
  \item If $s=2a$, then $s-1$ is odd. But we can multiply this generator of $I(\A)$, by any $y_1,\ldots,y_n$, to obtain via pairings similar as above (of course, making sure that if $y_j$ shows twice in a term, we don't pair it with itself), to obtain elements of $I(2,\A)$, of degree $a$.
\end{itemize}

We can see that a generator of $I(\A)$ either by itself is in $J(\A)$, or multiplied by a variable $y_j$; this gives the inclusion ``$\supseteq$''. For the inclusion ``$\subseteq$'', if $y_j\cdot F\in J(\A)\subseteq I(\A)$, and since $I(\A)$ is non-degenerate prime ideal (see \cite[Corollary 2.2]{ScTo}), then $F\in I(\A)$.

The last two statements are immediate from Proposition \ref{connection}.
\end{proof}

If $n=3$ and ${\rm rank}(\A)=3$, or if $n=2$ and ${\rm rank}(\A)=2$, since $I(\A)=0$ in both cases, then $I(2,\A)=0$ as well in both cases (we have $n<4$).

\subsection{Generators of the ideal of the Orlik-Terao algebra of the second order.} In the proof of Proposition \ref{properties} we give a glimpse of a standard way to find elements of $I(2,\A)$, that together with the elements obtained in Lemma \ref{standard_generators} will form a generating set for $I(2,\A)$.

Let us consider again the element $G_{i_1,\ldots,i_s}\in I(\A)$ corresponding to the circuit $\{i_1,\ldots,i_s\}, s\geq 3$. The whole idea is to multiply $G_{i_1,\ldots,i_s}$ by a convenient monomial $M\in{\bf S}$ (possibly $1$) such that in each term of this product, to be able to pair any two $y$'s with distinct indices (i.e., $MG_{i_1,\ldots,i_s}\in {\bf S'}$). Obviously $\deg(M)+s-1$ must be an even number, and if variable $y_a$ shows up in a term of $MG_{i_1,\ldots,i_s}$ with exponent $m_a$, since we cannot pair two $y_a$'s together, we must have $m_a\leq$ than the sum of the exponents of all the other variables in that term.

Often there will be a multitude of possible pairings, but the quadratic elements obtained in Lemma \ref{standard_generators} will help consider fewer. Nonetheless, as one can see in the example below, to find efficient ways to choose those monomials $M$ that will lead only to minimal generators of $I(2,\A)$, becomes a delicate technical problem.

\begin{exm}\label{example1} Let $\A=\{\ell_1,\ell_2,\ell_3,\ell_4\}\subset \mathbb K[x_1,x_2,x_3]$, with $$\ell_1=x_1,\ell_2=x_2,\ell_3=x_1+x_2,\ell_4=x_3.$$

We have $G:=G_{1,2,3}=y_2y_3+y_1y_3-y_1y_2\in I(\A)\subset {\bf S}=\mathbb K[y_1,y_2,y_3,y_4]$, and the standard generators of $I(2,\A)\subset{\bf T}=\mathbb K[t_{1,2},\ldots,t_{3,4}]$ exhibited in Lemma \ref{standard_generators} are:

\begin{eqnarray}
L_{1,2,3}&=& t_{2,3}+t_{1,3}-t_{1,2}\nonumber\\
Q_{1,2,3,4}^1&=&t_{1,2}t_{3,4}-t_{1,3}t_{2,4}\nonumber\\
Q_{1,2,3,4}^2&=&t_{1,2}t_{3,4}-t_{1,4}t_{2,3}\nonumber.
\end{eqnarray}

Let $M=y_1^{m_1}y_2^{m_2}y_3^{m_3}y_4^{m_4}$ be the monomial such that $MG\in{\bf S'}=\mathbb K[y_1y_2,\ldots,y_3y_4]$. Let $d:=m_1+m_2+m_3+m_4$. Then, $d$ must be an even number. We also must have
$$2m_1\leq d,\, 2m_2\leq d,\, 2m_3\leq d,\, 2m_4\leq d+2.$$

$\bullet$ If $m_4=0$, then $d=m_1+m_2+m_3$, with $m_1\leq m_2+m_3$, $m_2\leq m_1+m_3$, and $m_3\leq m_1+m_2$. So $M\in{\bf S'}$. Suppose $N$ is the preimage of $M$ in ${\bf T}$. Since we need a new generator, therefore different than $NL_{1,2,3}$, the only way would be to take (if possible) two distinct variables from $M$ and pair them ``differently'' with the variables in each term of $G$. Since we can suppose that $m_1\geq 1$ and $m_2\geq 1$, $$MG=(y_1y_2y_2y_3+y_1y_2y_1y_3-y_1y_2y_1y_2)\frac{M}{y_1y_2},$$ where the monomial $M/(y_1y_2)$ is assumed to be in ${\bf S'}$. But it is clear that in parenthesis we obtain $t_{1,2}L_{1,2,3}$, so no new generator.

$\bullet$ Suppose $m_4\geq 1$ and $d=2$. If $m_4=2$, then $MG$ gives the new minimal generator $$t_{2,4}t_{3,4}+t_{1,4}t_{3,4}-t_{1,4}t_{2,4}.$$ If $m_4=1$ and, say, $m_1=1$, then $MG=y_1y_4y_2y_3+y_1y_4y_1y_3-y_1y_4y_1y_2$. In the last two terms, since we cannot par $y_1$ with itself, we can only pair $y_1y_4$, $y_1y_3$, and $y_1y_2$. There are three different pairings we can do in the first term. But modulo the elements $Q_{1,2,3,4}^1$ and $Q_{1,2,3,4}^2$ we obtain $t_{1,4}L_{1,2,3}$, so no new minimal generator.

$\bullet$ Suppose $m_4\geq 1$. Suppose $d\geq 4$. Again, $M\in{\bf S'}$, and let $N$ be its preimage in ${\bf T}$. Since we need a new generator, so different than $NL_{1,2,3}$, the only way would be to take (if possible) two variables from $M$ and pair them ``differently'' with the variables in each term of $G$. From the first two bullets, the two variables  we pick should be $y_4$, and one of the other three. Since $d\geq 4$, then not all $m_1,m_2,m_3$ are zero. Suppose $m_1\geq 1$, and that we picked also $y_1$, and $M/(y_1y_4)\in{\bf S'}$. But the second bullet tells us that we do not get a new minimal generator for $I(2,\A)$.

In conclusion, $$I(2,\A)=\langle L_{1,2,3},Q_{1,2,3,4}^1,Q_{1,2,3,4}^2,t_{2,4}t_{3,4}+t_{1,4}t_{3,4}-t_{1,4}t_{2,4}\rangle.$$
\end{exm}

\subsection{The symmetric ideal of $I_{n-2}(\A)$ and Sylvester forms.} Denote $\mathbb T:=R[\ldots,t_{i,j},\ldots], 1\leq i<j\leq$, with the natural bi-grading: $\deg(x_u)=(1,0)$, and $\deg(t_{i,j})=(0,1)$. The {\em Rees algebra} of $I_{n-2}(\A)$, namely $R[I_{n-2}(\A)t]$, is isomorphic as bi-graded algebras to $\mathbb T/\mathcal I(\A,n-2)$, for some ideal $\mathcal I(\A,n-2)\subset \mathbb T$, called {\em the Rees ideal of $I_{n-2}(\A)$}, or {\em the presentation ideal of $R[I_{n-2}(\A)t]$}.

{\em The symmetric ideal of $I_{n-2}(\A)$} is the ideal of $\mathbb T$ generated by $\mathcal I(\A,n-2)_{(-,1)}$, and it will be denoted here by ${\rm sym}(I_{n-2}(\A))$.

We have that ${\rm sym}(I_{n-2}(\A))$ is generated by the linear generators obtained in Lemma \ref{standard_generators}, and from Theorem \ref{theorem1}, by the linear syzygies on the standard generators $f_{i,j}$ of $I_{n-2}(\A)$.

For any $1\leq a<b<c\leq n$ we have the standard syzygies $\ell_af_{a,b}=\ell_cf_{b,c}$, $\ell_af_{a,c}=\ell_bf_{b,c}$, and $\ell_bf_{a,b}=\ell_cf_{a,c}$, leading to
$$\underbrace{\ell_at_{a,b}-\ell_ct_{b,c}}_{A_{a,b,c}}, \underbrace{\ell_at_{a,c}-\ell_bt_{b,c}}_{B_{a,b,c}},  \underbrace{\ell_bt_{a,b}-\ell_ct_{a,c}}_{C_{a,b,c}}\in {\rm sym}(I_{n-2}(\A)).$$

\begin{prop}\label{gens_sym_ideal} The symmetric ideal ${\rm sym}(I_{n-2}(\A))$ is generated by all $L_{i_1,i_2,i_3}$, whenever $\{i_1,i_2,i_3\}$ is a circuit, and by $A_{a,b,c},B_{a,b,c},C_{a,b,c}$, for all $1\leq a<b<c\leq n$.
\end{prop}
\begin{proof}
Let $$\sum_{1\leq i<j\leq n}h_{i,j}f_{i,j}=0,$$ be a linear syzygy. Then the vector $(\ldots,\widehat{h}_{i,j},\ldots)\in {\rm ker}(\phi_{\A})$. By Lemma \ref{kernel} and Claim 1, this vector a combination of 3-dependencies, with coefficients linear forms. This means that this syzygy is a combination of the ``syzygies'' corresponding to the 3-dependencies, modulo $R/\langle\ell_i,\ell_j\rangle$ in each $(i,j)$ entry.

Since we already accounted for the ``syzygies'' corresponding to the 3-dependencies, we can assume that each $h_{i,j}\in \langle\ell_i,\ell_j\rangle$. Suppose for all $1\leq i<j\leq n$, $$h_{i,j}=\alpha_{i,j}\ell_i+\beta_{i,j}\ell_j, \alpha_{i,j},\beta_{i,j}\in\mathbb K.$$ These plugged back in the syzygy equation lead to a syzygy on $f_1:=f/\ell_1,\ldots,f_n:=f/\ell_n$, the generators of $I_{n-1}(\A)$

$$\sum_{i=1}^n\left(\sum_{u=1}^{i-1}\alpha_{u,i}+\sum_{v=i+1}^n\beta_{i,v}\right)f_i=0.$$ Since $\gcd(\ell_i,\ell_j)=1,i\neq j$, and since $\alpha_{i,j},\beta_{i,j}\in\mathbb K$, we have that for each $i=1,\ldots,n$,

$$\sum_{u=1}^{i-1}\alpha_{u,i}+\sum_{v=i+1}^n\beta_{i,v}=0.$$

With these equations we can rewrite the syzygy in the following way:

\begin{eqnarray}
\beta_{1,2}(\ell_2f_{1,2}-\ell_nf_{1,n})+\beta_{1,3}(\ell_3f_{1,3}-\ell_nf_{1,n})+&\cdots&+\beta_{1,n-1}(\ell_{n-1}f_{1,n-1}-\ell_nf_{1,n})+\nonumber\\
\alpha_{1,2}(\ell_1f_{1,2}-\ell_nf_{2,n})+\beta_{2,3}(\ell_3f_{2,3}-\ell_nf_{2,n})+&\cdots&+\beta_{2,n-1}(\ell_{n-1}f_{2,n-1}-\ell_nf_{2,n})+\nonumber\\
&\vdots&\nonumber\\
\alpha_{1,n}(\ell_1f_{1,n}-\ell_{n-1}f_{n-1,n})+\alpha_{2,n}(\ell_2f_{2,n}-\ell_{n-1}f_{n-1,n})+&\cdots&+\alpha_{n-2,n}(\ell_{n-2}f_{n-2,n}-\ell_{n-1}f_{n-1,n}).\nonumber
\end{eqnarray}

But this expression confirms that the syzygy can be written as a combination of the standard syzygies. So the result is shown.
\end{proof}

\begin{rem} In Theorem \ref{theorem1}, the dimension of standard syzygies equals $\displaystyle 2{{n}\choose{2}}-n-2p(\A)=n(n-2)-2p(\A)$. Since $p(\A)$ is the number of minimal linear generators of ${\rm sym}(I_{n-2}(\A))$, we obtain that the minimum number of generators of ${\rm sym}(I_{n-2}(\A))$ is $n(n-2)-p(\A)$.
\end{rem}

\subsubsection{Sylvester forms.} The {\em Sylvester forms} technique (see for example \cite[Section 2]{HoSiVa1}) is a nice way to find new elements of the Rees ideal, from old elements. This technique was successfully applied in \cite[Proposition 3.5]{GaSiTo}, to obtain all the generators of $I(\A)$, the Orlik-Terao ideal, and we will do the same to obtain elements in $I(2,\A)$, from generators of ${\rm sym}(I_{n-2}(\A))$.

Since the level of computations exceeds the plans of this article, we are just going to exemplify them for some basic situations.

\medskip

$\bullet$ Suppose $\ell_3=a_1\ell_1+a_2\ell_2$ is a 3-dependency; of course $\ell_1,\ell_2$ will form a regular sequence in $R$. If we consider $A_{1,2,3}=\ell_1t_{1,2}-\ell_3t_{2,3}=\ell_1(t_{1,2}-a_1t_{2,3})+\ell_2(-a_2t_{2,3})$ and $B_{1,2,3}=\ell_1t_{1,3}+\ell_2(-t_{2,3})$, then we have the matrix equation

$$\left[\begin{array}{l} A_{1,2,3}\\ B_{1,2,3}
\end{array}\right]=\left[
\begin{array}{cc}
t_{1,2}-a_1t_{2,3}&-a_2t_{2,3}\\
t_{1,3}&-t_{2,3}
\end{array}
\right]\cdot \left[\begin{array}{l} \ell_1\\ \ell_2
\end{array}\right].$$ Taking the determinant of the $2\times 2$ {\em content} matrix we obtain $$-t_{2,3}(t_{1,2}-a_1t_{2,3}-a_2t_{1,3})=\pm t_{2,3}L_{1,2,3}$$ as an element of $I(2,\A)$. Since $I(2,\A)$ is prime, we obtain the linear generator $L_{1,2,3}\in I(2,\A)$.

\medskip

$\bullet$ Suppose $\ell_4=a_1\ell_1+a_2\ell_2+a_3\ell_3$ is a dependency, with $\ell_1,\ell_2,\ell_3$ being linearly independent. With this dependency, and choosing $A_{1,2,3},B_{1,2,3},A_{1,2,4}$ we have the matrix equation

$$\left[\begin{array}{l} A_{1,2,3}\\ B_{1,2,3}\\A_{1,2,4}
\end{array}\right]=\left[
\begin{array}{ccc}
t_{1,2}&0&-t_{2,3}\\
t_{1,3}&-t_{2,3}&0\\
t_{1,2}-a_1t_{2,4}&-a_2t_{2,4}&-a_3t_{2,4}
\end{array}
\right]\cdot \left[\begin{array}{l} \ell_1\\ \ell_2\\\ell_3
\end{array}\right].$$ Taking the determinant of the $3\times 3$ content matrix we obtain

$$t_{2,3}\underbrace{(a_3t_{1,2}t_{2,4}+a_2t_{1,3}t_{2,4}+a_1t_{2,3}t_{2,4}-t_{1,2}t_{2,3})}_{F}\in I(2,\A).$$ So $F\in I(2,\A)$.

If $G=a_1y_2y_3y_4+a_2y_1y_3y_4+a_3y_1y_2y_4-y_1y_2y_3$ is the generators of $I(\A)$ corresponding to the given dependency, then $F$ is the preimage in ${\bf T}$ of $$y_2G=a_1(y_2y_3)(y_2y_4)+a_2(y_1y_3)(y_2y_4)+a_3(y_1y_2)(y_2y_4)-(y_1y_2)(y_2y_3).$$

\medskip

$\bullet$ Even the quadratic standard generators of $I(2,\A)$ from Lemma \ref{standard_generators} can be obtained via Sylvester forms. Suppose $\A=\{\ell_1,\ell_2,\ldots,\ell_n\}$ with $n\geq 4$. Using $A_{1,2,4}=\ell_1t_{1,2}-\ell_4t_{2,4}$ and $A_{1,3,4}=\ell_1t_{1,3}-\ell_4t_{3,4}$, and the fact that $\ell_1,\ell_4$ are linearly independent, we get the matrix equation
$$\left[\begin{array}{l} A_{1,2,4}\\ A_{1,3,4}
\end{array}\right]=\left[
\begin{array}{cc}
t_{1,2}&-t_{2,4}\\
t_{1,3}&-t_{3,4}
\end{array}
\right]\cdot \left[\begin{array}{l} \ell_1\\ \ell_4
\end{array}\right].$$ The determinant of the content matrix  is $t_{1,3}t_{2,4}-t_{1,2}t_{3,4}=-Q_{1,2,3,4}^1$.

\subsection{Comments on Cohen-Macaulayness of $OT(\A,2)$.} In \cite{PrSp} it is proven that $OT(\A)$ is (arithmetically) Cohen-Macaulay. This result is also recovered in \cite{GaSiTo}, by the means of Rees algebra, and special fiber related results. With this late approach, in the spirit of \cite{DoRaSi}, there is the hope that one can prove that $OT(2,\A)$ is also Cohen-Macaulay, at least for the case when $k=3$ (in three variables). Indeed, $I_{n-2}(\A)$ is linearly presented (from Theorem \ref{theorem1}), and it is generated by the maximal minors of a $(n-2)\times n$ matrix with linear forms entries (see the last paragraphs of the proof of \cite[Proposition 2.1]{To2}). But the ideals $I\subset A:=\mathbb K[x,y,z]$ considered in \cite{DoRaSi} are perfect ideals, causing for $A/I$ to be Cohen-Macaulay. In our situation, if $k=3$, $R/I_{n-2}(\A)$ is Cohen-Macaulay if and only if $\A$ is a rank 3 generic hyperplane arrangement (see Corollary \ref{CM}). In these conditions, by the same coding theory argument we presented in Section \ref{3vars}, one has that $$I_{n-2}(\A)=\langle x_1,x_2,x_3\rangle^{n-2}\subset R=\mathbb K[x_1,x_2,x_3].$$

The special fiber of any power $d$ of the maximal ideal ${\frak m}$ of any ring of homogeneous polynomials with coefficients in a field $\mathbb K$ is the $d$-th Veronese algebra. By \cite[Theorem 5]{Ba} (citing Gr\"{o}bner, \cite{Gr}), it is (arithmetically) Cohen-Macaulay. From this, with \cite[Theorem 4.44]{Va}, we get furthermore, that the Rees algebra of ${\frak m}^d$ is also Cohen-Macaulay. Similar argument shows that if $\A$ is any arrangement of rank 2, then the special fiber and the Rees algebra of $I_{n-2}(\A)$ are Cohen-Macaulay (see the proof of Claim 3 in Section \ref{the_proof}).

Despite that this approach leads to very special cases, we are still conjecturing that for any $\A$ of any rank $\geq 2$, $OT(2,\A)$ is Cohen-Macaulay.

\renewcommand{\baselinestretch}{1.0}
\small\normalsize 

\bibliographystyle{amsalpha}

\end{document}